\documentclass[a4paper,11pt]{amsart}
\usepackage{amsmath,amsthm,amssymb,amscd}

\title{Flat Mittag-Leffler modules over countable rings}

\author{Silvana Bazzoni}
\address{
Dipartimento di Matematica Pura e Applicata, Universit\'a di Padova \\
Via Trieste 63, 35121 Padova, Italy
}
\email{bazzoni@math.unipd.it}

\author{Jan \v{S}\v{t}ov\'\i\v{c}ek}
\address{
Charles University in Prague, Faculty of Mathematics and Physics \\
Department of Algebra \\
Sokolovska 83, 186 75 Praha 8, Czech Republic
}
\email{stovicek@karlin.mff.cuni.cz}


\thanks{%
First named author supported by MIUR, PRIN 2007, project ``Rings, algebras, modules and categories'' and by Universit\`{a} di Padova (Progetto di Ateneo CPDA071244/07 `` Algebras and cluster categories'').
Second named author supported by the Eduard \v{C}ech Center for Algebra and Geometry (LC505).
}
\date{\today}

\renewcommand{\iff}{if and only if }
\newcommand{\st}{such that }

\newcommand{\la}{\longrightarrow}
\newcommand{\li}{\varinjlim}

\DeclareMathOperator{\Hom}{Hom}

\DeclareMathOperator{\Ext}{Ext}

\DeclareMathOperator{\Img}{Im}

\newcommand{\Mod}[1]{\hbox{\rm Mod-}{#1}}
\newcommand{\Flat}[1]{\hbox{\rm Flat-}{#1}}

\newcommand{\C}{\mathcal{C}}
\newcommand{\D}{\mathcal{D}}

\newcommand{\clS}{\mathcal{S}}

\newcommand{\X}{\mathcal{X}}

\newcommand{\card}[1]{\left\lvert{#1}\right\rvert}

\theoremstyle{plain}
\newtheorem{thm}{Theorem}
\newtheorem{lem}[thm]{Lemma}
\newtheorem{prop}[thm]{Proposition}

\theoremstyle{definition}
\newtheorem{defn}[thm]{Definition}

\theoremstyle{remark}
\newtheorem{rem}[thm]{Remark}

\begin{document}


\begin{abstract}
We show that over any ring, the double Ext-orthogonal class to all flat Mittag-Leffler modules contains all countable direct limits of flat Mittag-Leffler modules. If the ring is countable, then the double orthogonal class consists precisely of all flat modules and we deduce, using a recent result of \v{S}aroch and Trlifaj, that the class of flat Mittag-Leffler modules is not precovering in $\Mod{R}$ unless $R$ is right perfect.
\end{abstract}

\maketitle

\section*{Introduction}

The notion of a Mittag-Leffler module was introduced by Raynaud and Gruson~\cite{RG}, who used the concept to prove a conjecture due to Grothendieck that the projectivity of infinitely generated modules over commutative rings is a local property. This is a crucial step for defining and working with infinitely generated vector bundles, as considered by Drinfeld in~\cite{Dr}, where we also refer for more explanation.

The main step behind this geometrically motivated result is a completely general characterization of projective modules over any (in general non-commutative) ring $R$. Namely, one can show (consult~\cite{Dr}) that an $R$-module is projective \iff $M$ satisfies the following three conditions:

\begin{enumerate}
 \item $M$ is flat,
 \item $M$ is Mittag-Leffler,
 \item $M$ is a direct sum of countably generated modules.
\end{enumerate}

As mentioned by Drinfeld, the proof of projectivity of a given module might be non-constructive even in very simple cases, because it requires the Axiom of Choice. This applies for instance to the ring $\mathbb{Q}$ of rational numbers and the $\mathbb{Q}$-module $\mathbb{R}$ of real numbers. The main trouble there is condition (3). Thus, one might consider replacing projective modules by flat Mittag-Leffler modules (these are called ``projective modules with a human face'' in a preliminary version of~\cite{Dr}).

However, a surprising result in~\cite[\S5]{EGPT} indicates that if one is interested in homological algebra, this might not be a good idea at all. Namely, the class of flat Mittag-Leffler abelian groups does not provide for precovers (sometimes also called right approximations). In the present paper, we use recent results due to \v{S}aroch and Trlifaj~\cite{SarT} to show this is a much more general phenomenon and applies to many geometrically interesting examples. Namely, we prove in Theorem~\ref{thm:double_perp_of_flat_ML} that the class of flat Mittag-Leffler $R$-modules over a countable ring $R$ is precovering \iff $R$ is a right perfect ring. Note that in that case the classes of projective modules, flat Mittag-Leffler modules and flat modules coincide, so the flat Mittag-Leffler precovers are just the projective ones.

\section{Preliminaries}

In this paper, $R$ will always be an associative, not necessarily commutative ring with a unit. If not specified otherwise, a module will stand for a right $R$-module. We will denote by $\D$ the class of all modules which are flat and satisfy the Mittag-Leffler condition in the sense of~\cite{RG,HT}:

\begin{defn} \label{def:ML}
$M$ is called a \emph{Mittag-Leffler} module if the canonical morphism
$ \rho: M \otimes_R \prod_{i \in I} Q_i \la \prod_{i \in I} M \otimes_R Q_i $
is injective for each family $(Q_i \mid i \in I)$ of left $R$-modules.
\end{defn}

A crucial closure property of the class $\D$ has been obtained in~\cite{HT}:

\begin{prop} \label{prop:aleph_1_continuous_systems} \cite[Proposition 2.2]{HT}
Let $R$ be a ring and $(F_i, u_{ji}: F_i \to F_j)$ be a direct system of modules from $\D$ indexed by $(I,\leq)$. Assume that for each increasing chain $(i_n \mid n < \omega)$ in $I$, the module $\li_{n < \omega} F_{i_n}$ belongs to $\D$. Then $M = \li_{i \in I} F_i$ belongs to $\D$.
\end{prop}

Let us look more closely at countable chains of modules and their limits. Recall that given a sequence of morphisms
$$ F_0 \overset{u_0}\la F_1 \overset{u_1}\la F_2 \overset{u_2}\la F_3 \la \dots $$
we have a short exact sequence
$$
\eta: \qquad
0 \la \bigoplus_{n<\omega} F_n \overset{\varphi}\la \bigoplus_{n<\omega} F_n \la \li F_n \la 0,
\eqno{(*)} \label{eqn:countable_lim}
$$
\st $\varphi$ is defined by $\varphi \iota_n = \iota_n - \iota_{n+1} u_n$, where $\iota_n: F_n \to \bigoplus F_m$ are the canonical inclusions. Note the following simple fact:

\begin{lem} \label{lem:local_splitting_of_lim}
Given a chain $ F_0 \overset{u_0}\to F_1 \overset{u_1}\to F_2 \to \dots$ of morphisms as above and a number $n_0 < \omega$, the middle term of~$(*)$ decomposes as
$$ \bigoplus_{n<\omega} F_n = \varphi\left(\bigoplus_{m<n_0} F_m\right) \oplus \left(\bigoplus_{m\ge n_0} F_m\right). $$
\end{lem}

\begin{proof}
Note that the module $\varphi(\bigoplus_{m<n_0} F_m)$ is generated by elements of the form
$$
\left(0, \dots, 0, x_m, -u_m(x_m), 0 \dots\right) \in \bigoplus_{n<\omega} F_n,
$$
where $m < n_0$ and $x_m \in F_m$. It follows easily that one can uniquely express each $y = (y_n) \in \bigoplus_{n<\omega} F_n$ as $y = z+w$, where $z \in \varphi\left(\bigoplus_{m<n_0} F_m\right)$ and $w \in \bigoplus_{m\ge n_0} F_m$. Namely, we take 
$z = (y_0, \dots, y_{n_0-1},y'_{n_0}, 0, 0, \dots)$ with $-y'_{n_0}=u_{n_0-1}(y_{n_0-1}) + u_{n_0-1}u_{n_0-2}(y_{n_0-2}) + \dots + u_{n_0-1}u_{n_0-2}\dots u_{1}u_{0}(y_0)$ and $w = y-z \in \bigoplus_{m\ge n_0} F_m$.
\end{proof}

We will also need a few simple results concerning infinite combinatorics, starting with a well known lemma.

\begin{lem} \label{lem:nice_cardinals}
For any cardinal $\mu$ there is a cardinal $\lambda \ge \mu$ \st $\lambda^{\aleph_0} = 2^\lambda$.
\end{lem}

\begin{proof}
We refer for instance to~\cite[Lemma 3.1]{Gr}. For the reader's convenience, we recall how to construct $\lambda$. We put $\mu_0 = \mu$ and for each $n<\omega$ inductively construct $\mu_{n+1} = 2^{\mu_n}$. Then $\lambda = \sup_{n<\omega} \mu_n$ has the required property, see for example~\cite[p. 50, fact (6.21)]{Jech}.
\end{proof}

The next lemma deals with a construction of a large family of ``almost disjoint'' maps $f: \omega \to \lambda$. The result is
well known in the literature and it has many different proofs. We refer for instance to ~\cite[Lemma 2.3]{Ba} or ~\cite[Proposition II.5.5]{EM}.

\begin{lem} \label{lem:almost_disjoint} 
Let $\lambda$ be an infinite cardinal. Then there is a subset $J \subseteq \lambda^\omega$ of cardinality $\lambda^{\aleph_0}$ \st for any pair of distinct maps $f,g: \omega \to \lambda$ of $J$, the set formed by the $x \in \omega$ on which the values $f(x)$ and $g(x)$ coincide is a finite initial segment of $\omega$.
\end{lem}

\begin{proof}
Consider the tree $T$ of the finite sequences of elements of $\lambda$, i.e. $T = \{t: n \to \lambda \mid n<\omega\}$. Since $\lambda$ is infinite, we have $\card{T} = \card{\bigcup_{n<\omega} \lambda^n} = \lambda$, so there is a bijection $b: T \to \lambda$. For every map $f: \omega \to \lambda$ denote by $A_f: \omega \to T$ the induced map which sends $n<\omega$ to $f\restriction n \in T$. Clearly, if $f$, $g$ are two different maps in $\lambda^\omega$, the values of 
$A_f$ and $A_g$ coincide only on a finite initial segment of $\omega$. Now, we can put $J = \{ b\circ A_f \mid f \in \lambda^\omega \}$.
\end{proof}

\section{Main result}

Now we are in a position to state our main result, which is inspired by~\cite[\S5]{HT}. It will be proved by using a cardinal argument similar to the one in ~\cite[Proposition 2.5]{Ba}. Note that the result sharpens~\cite[Theorem 2.9]{SarT} by removing the additional set-theoretical assumption of Singular Cardinal Hypothesis, and also~\cite[Corollaries 7.6 and 7.7]{HT} by removing the assumption that $\D$ is closed under products.

Regarding the notation and terminology, given a class $\C \subseteq \Mod{R}$, we put $\C^\perp = \{ M \in \Mod{R} \mid \Ext^1_R(\C,M) = 0 \}$ and dually ${^\perp\C} = \{ M \in \Mod{R} \mid \Ext^1_R(M,\C) = 0 \}$. Recall that a module is called \emph{cotorsion} if it cannot be non-trivially extended by a flat module.

We recall also the notion of a precover, or sometimes called right approximation. If $\X$ is any class of modules and $M\in \Mod{R}$, a homomorphism $f\colon X \to M$ is called an $\X$-\emph{precover} of $M$ if $X \in \X$ and for every homomorphism $f '\in \Hom_R(X', M)$ with $X'\in\X$ there exists a homomorphism $g\colon X'\to X$ such that $f '= f g$. The class $\X$ is called \emph{precovering} if each $M \in \Mod{R}$ admits an $\X$-precover.

\begin{thm} \label{thm:double_perp_of_flat_ML}
Let $R$ be a ring and $\D$ be the class of all flat Mittag-Leffler right $R$-modules. Given any countable chain
$$ F_0 \overset{u_0}\la F_1 \overset{u_1}\la F_2 \overset{u_2}\la F_3 \la \dots $$
of morphisms \st $F_n \in \D$ for all $n<\omega$, we have $\li F_n \in {^\perp (\D^\perp)}$.
If, moreover, $R$ is a countable ring, then the following hold:
\begin{enumerate}
 \item $\D^\perp$ is precisely the class of all cotorsion modules.
 \item $\D$ is a precovering class in $\Mod{R}$ \iff $R$ is right perfect.
\end{enumerate}
\end{thm}

\begin{proof}
Assume we have a countable direct system $(F_n, u_n)$ as above, put $F = \li F_n$, and fix a module $C \in \D^\perp$. We must prove that $\Ext^1_R(F,C) = 0$.

Let us fix an infinite cardinal $\lambda$, depending on $C$, \st we have $\lambda \ge \card{\Hom_R(F_n,C)}$ for each $n<\omega$ and $\lambda^{\aleph_0} = 2^\lambda$; we can do this using Lemma~\ref{lem:nice_cardinals}. Applying Lemma~\ref{lem:almost_disjoint}, we find a subset $J \subseteq \lambda^\omega$ of cardinality $2^\lambda$ \st the values of each pair $f,g: \omega \to \lambda$ of distinct elements of $J$ coincide only on a finite initial segment of~$\omega$. We claim that there is a short exact sequence of the form
$$
0 \la P \la E  \la F^{(2^\lambda)} \la 0
\eqno{(\dag)} \label{eqn:copies_of_F_loosely_together}
$$
\st $E \in \D$ and $\card{\Hom_R(P,C)} \le 2^\lambda$.

Let us construct such a sequence. First, denote for each $\alpha<\lambda$ by $F_{n,\alpha}$ a copy of $F_n$, and by $P$ the direct sum $\bigoplus F_{n,\alpha}$ taken over all pairs $(n,\alpha)$ \st $n<\omega$ and $\alpha=f(n)$ for some $f \in J$. Note that $P$ is a summand in $\bigoplus_{n<\omega} F_n^{(\lambda)}$, so we have
$$
\card{\Hom_R(P,C)} \le
\card{\Hom_R\big(\bigoplus F_n^{(\lambda)},C\big)} \le
\prod_{n<\omega} \card{\Hom_R(F_n,C)}^\lambda \le
\lambda^{\omega \times \lambda} =
2^\lambda.
$$

Next, we will construct $E$. Given $f \in J$, let
$$
\iota_f: \bigoplus_{n<\omega} F_n \la P
$$
be the split inclusion which sends each $F_n$ to $F_{n,f(n)}$. Using the short exact sequence $(*)$ from page~\pageref{eqn:countable_lim}, we can extend $P$ by $F$ via the following pushout diagram:
$$
\begin{CD}
\eta: \qquad
@.         0   @>>>   \bigoplus_{n<\omega} F_n   @>{\varphi}>>   \bigoplus_{n<\omega} F_n   @>>>   F   @>>>   0   \\
     @.    @.                  @V{\iota_f}VV                       @V{\vartheta_f}VV               @|             \\
\varepsilon_f: \qquad
@.         0   @>>>                  P          @>{\subseteq}>>           E_f               @>>>   F   @>>>   0
\end{CD}
\eqno{(\Delta)} \label{eqn:construction_E_f}
$$

Now, we can put these extensions for all $f \in J$ together. Namely, let $\sigma: P^{(J)} \to P$ be the summing map and consider the pushout diagram:
$$
\begin{CD}
\bigoplus \varepsilon_f: \qquad
@.   0   @>>>   P^{(J)}   @>>>   \bigoplus_{f \in J} E_f   @>{\rho}>>   F^{(J)}   @>>>   0   \\
  @. @.      @V{\sigma}VV               @V{\pi}VV                   @|                 \\
\varepsilon: \qquad
@.   0   @>>>      P      @>>>              E              @>>>   F^{(J)}   @>>>   0
\end{CD}
$$
%
For each $g \in J$, the composition of the canonical inclusion $\nu_g\colon E_g \to \bigoplus_{f \in J} E_f$ with the morphism $\pi$ yields a monomorphism $E_g \to E$. In fact, if $y\in E_g$ is such that $\pi \nu_g(y)=0$, then $\rho(\nu_g(y))=0$, hence the exact sequence $\varepsilon_g$ gives that $y$ is in the image of $P$ and the composition of the canonical embedding $\mu_g\colon P \to P^{(J)}$ with the morphism $\sigma$ is a monomorphism. From now on we shall without loss of generality view these monomorphisms $E_g \to E$ as inclusions.

To prove the existence of $(\dag)$, it suffices to show that $E \in \D$ in $\varepsilon$. To this end, denote for any subset $S \subseteq J$ by $M_S$ the module
$$ M_S = \sum_{f \in S} \Img \vartheta_f \quad (\subseteq E, \textrm{ see diagram } (\Delta)) $$
Then the family $(M_S \mid S \subseteq J \,\&\, \card{S} \le \aleph_0)$ with obvious inclusions forms a direct system and we claim that its union is the whole of $E$. Indeed, it is straightforward to check, using diagram~$(\Delta)$ and the construction of the embeddings $E_g \subseteq E$, that $E = P + \sum_{f \in J} \Img \vartheta_f$. Further, the left hand square of diagram~$(\Delta)$ is a pull-back, which implies $P \cap \Img \vartheta_f = \Img \iota_f$ and
$$ P \cap \sum_{f \in J} \Img \vartheta_f \supseteq \sum_{f \in J} \Img \iota_f = P. $$
Thus, $E = \sum_{f \in J} \Img \vartheta_f$ and the claim is proved.

Moreover, the union of any chain $M_{S_0} \subseteq M_{S_1} \subseteq M_{S_2} \subseteq \dots$ from the direct system belongs to the direct system again. Therefore, if we prove that $M_S \in \D$ for each countable $S \subseteq J$, it will follow from Proposition~\ref{prop:aleph_1_continuous_systems} that $E \in \D$. Our task is then reduced to prove the following lemma:

\begin{lem} \label{lem:inner_structure_of_E}
With the notation as above, the following hold:
\begin{enumerate}
\item Given $S \subseteq T \subseteq J$ with $S$ and $T$ finite and \st $\card{T} = \card{S}+1$, the inclusion $M_S \subseteq M_T$ splits and there is $n_0 < \omega$ \st $M_T/M_S \cong \bigoplus_{m \ge n_0} F_m$.
\item Given a countable subset $S \subseteq J$, the module $M_S$ is isomorphic to a countable direct sum with each summand isomorphic to some $F_n$, $n<\omega$. In particular, $M_S \in \D$.
\end{enumerate}
\end{lem}

\begin{proof}
Let us focus on (1) since (2) is an immediate consequence. Denote by $f: \omega \to \lambda$ the single element of $T \setminus S$, and let $n_0 < \omega$ be the smallest number \st $f(n_0) \ne g(n_0)$ for each $g \in S$.

We claim that the following are satisfied by the construction:
$$
M_S \cap \Img \vartheta_f =
\iota_f\left(\bigoplus_{m<n_0} F_m\right) =
\vartheta_f \circ \varphi\left(\bigoplus_{m<n_0} F_m\right).
$$
The second equality holds simply because $\vartheta_f \circ \varphi = \iota_f$ by diagram~$(\Delta)$. For the first, note that 
$M_S \cap \Img \vartheta_f $
as a submodule of $E$, is contained in $P$. Since $P \cap \Img \vartheta_f = \Img \iota_f$, we have
$$ M_S \cap \Img \vartheta_f =\left(\sum_{g \in S} \Img \vartheta_g\right) \cap \Img \iota_f =\left(\sum_{g \in S} \Img \iota_g\right) \cap \Img \iota_f = \iota_f\left(\bigoplus_{m<n_0} F_m\right), $$
by the construction of $P$. This proves the claim.

Invoking Lemma~\ref{lem:local_splitting_of_lim}, we further deduce that
$$ \Img \vartheta_f = \iota_f\left(\bigoplus_{m<n_0} F_m\right) \oplus \vartheta_f\left(\bigoplus_{m\ge n_0} F_m\right). $$
In particular, the inclusion $M_S \cap \Img \vartheta_f \subseteq \Img \vartheta_f$ splits and so does the inclusion $M_S \subseteq M_S + \Img \vartheta_f = M_T$. Moreover, we have the isomorphisms
$$
M_T/M_S = (M_S + \Img \vartheta_f)/M_S \cong \Img \vartheta_f/(M_S \cap \Img \vartheta_f) \cong \bigoplus_{m\ge n_0} F_m,
$$
which finishes the proof of the lemma.
\end{proof}

Having established the existence of $(\dag)$ \st $E \in \D$ and $\card{\Hom_R(P,C)} \le 2^\lambda$, let us apply $\Hom_R(-,C)$ on $(\dag)$. Since $C \in \D^\perp$ by assumption, we get an exact sequence
$$
\Hom_R(P,C) \la \Ext^1_R(F^{(2^\lambda)},C) \la 0.
$$
Suppose now that $\Ext^1_R(F,C) \ne 0$. Then we would have $\lvert \Ext^1_R(F^{(2^\lambda)},C) \rvert \ge 2^{2^\lambda}$, which would contradict the fact that $\card{\Hom_R(P,C)} \le 2^\lambda$. Hence $\Ext^1_R(F,C) = 0$ as desired.

To finish the proof of Theorem~\ref{thm:double_perp_of_flat_ML}, suppose $R$ is a countable ring. Since each $F \in \D$ is flat, $\D^\perp$ contains all cotorsion modules. On the other hand, if $C$ is not cotorsion, there is a countable flat module $F$ \st $\Ext^1_R(F,C) \ne 0$; see for instance~\cite[Theorems 4.1.1 and 3.2.9]{GT}. By the first part of Theorem~\ref{thm:double_perp_of_flat_ML}, we know that $F \in {^\perp (\D^\perp)}$, so $C \not\in \D^\perp$. Hence $\D^\perp$ consists precisely of cotorsion modules.

The fact that $\D$ is not precovering unless $R$ is right perfect (and $\D$ is then the class of projective modules) follows directly from~\cite[Theorem 2.10]{SarT}. This finishes the proof of Theorem~\ref{thm:double_perp_of_flat_ML}.
\end{proof}

\begin{rem} \label{rem:constructive}
The proof of Theorem~\ref{thm:double_perp_of_flat_ML} is to some extent constructive. Namely, if $R$ is a countable ring and $C$ is a module which is not cotorsion, the theorem gives us a recipe how to construct $E \in \D$ \st $\Ext^1_R(E,C) \ne 0$, and it allows us to estimate the size of $E$ based on the size of $C$. Note that if $R$ is non-perfect, the size of $E$ must grow with the size of $C$. This is because for any set $\clS \subseteq \D$, we have ${^\perp (\clS^\perp)} \subseteq \D \subsetneqq \Flat{R}$ by~\cite[Proposition 1.9]{AH} and~\cite[Corollary 3.2.3]{GT}.
\end{rem}


\end{document}